\documentclass{amsart}
\usepackage{pst-plot}
\usepackage{graphicx}
\usepackage{epstopdf}
\usepackage[english]{babel}
\usepackage[latin1]{inputenc}
 \usepackage[all]{xy}

\usepackage{amsmath,amsfonts,amssymb,amsthm,amscd,array,stmaryrd,mathrsfs, mathdots}
\PassOptionsToPackage{option}{xcolor}
\usepackage{pstricks}

\theoremstyle{plain}
\newtheorem{thm}{Theorem}[section]
\newtheorem{lem}{Lemma}[section]

\newtheorem{prop}[lem]{Proposition}
\newtheorem{conj}[thm]{Conjecture}
\newtheorem{theo}{Theorem}

\theoremstyle{definition}
\newtheorem{defn}[lem]{Definition}
\newtheorem{rem}[lem]{Remark}
\newtheorem{ex}[lem]{Example}

\let\ssection=\section
\renewcommand{\section}{\setcounter{equation}{0}\ssection}

\newcommand{\R}{\mathbb{R}}
\newcommand{\Z}{\mathbb{Z}}
\newcommand{\C}{\mathbb{C}}
\newcommand{\N}{\mathbb{N}}

\newcommand{\F}{\mathcal{F}}

\newcommand{\Pc}{\mathcal{P}} 
 
\newcommand{\Rc}{\mathcal{R}}
\newcommand{\Sc}{\mathcal{S}}

\newcommand{\Id}{\mathrm{Id}}

\newcommand{\SL}{\mathrm{SL}}
\newcommand{\PSL}{\mathrm{PSL}}

\begin{document}

\title{On radius of convergence of $q$-deformed real numbers}

\author[L. Leclere]{Ludivine Leclere}
\author[S. Morier-Genoud]{Sophie Morier-Genoud}
\author[V. Ovsienko]{Valentin Ovsienko}
\author[A. Veselov]{Alexander Veselov}

\address{Ludivine Leclere, Sophie Morier-Genoud,
Laboratoire de Math\'ematiques,
Universit\'e de Reims,
U.F.R. Sciences Exactes et Naturelles,
Moulin de la Housse - BP 1039,
51687 Reims cedex 2,
France
}

\address{
Valentin Ovsienko,
Centre National de la Recherche Scientifique,
Laboratoire de Math\'ematiques,
Universit\'e de Reims,
U.F.R. Sciences Exactes et Naturelles,
Moulin de la Housse - BP 1039,
51687 Reims cedex 2,
France}

\address{
Alexander P. Veselov,
Department of Mathematical Sciences, 
Loughborough University, 
Loughborough, LE11 3TU, UK}

\email{ludivine.lecler@univ-reims.fr,
sophie.morier-genoud@univ-reims.fr,
valentin.ovsienko@univ-reims.fr,
A.P.Veselov@lboro.ac.uk}

\begin{abstract}
We study analytic properties of ``$q$-deformed real numbers'',
a notion recently introduced by two of us.
A $q$-deformed positive real number is a power series with integer coefficients in one formal variable~$q$.
We study the radius of convergence of these power series assuming that $q$ is a complex variable.
Our main conjecture, which can be viewed as a $q$-analogue of Hurwitz's Irrational Number Theorem,  claims that the $q$-deformed golden ratio has the smallest radius of convergence among all real numbers.
The conjecture is proved for certain class of rational numbers and confirmed by a number of computer experiments.
We also prove the explicit lower bounds for the radius of convergence for the $q$-deformed convergents of golden and silver ratios.
\end{abstract}

\maketitle

\thispagestyle{empty}

\section{Introduction}\label{IntroS}

There is a famous result due to Hurwitz \cite{Hurwitz}, which roughly claims that the golden ratio $\varphi=\frac{1+\sqrt{5}}{2}$ is the most irrational number.
More precisely, for any real number $x \in \R$ one can define as the measure of irrationality the Markov constant $\mu(x)$, which is the infimum of $c$, for which the inequality $|x-\frac{p}{q}|<\frac{c}{q^2}$ holds for infinitely many integer $p,q.$ Hurwitz's Irrational Number Theorem claims that for every $x \in \R$
$$
\mu(x)\leq \mu(\varphi)=\frac{1}{\sqrt{5}}
$$
with equality holding only for $x$ which are $\PSL(2,\Z)$-equivalent to the golden ratio.

In this paper we discuss a possible $q$-analogue of this classical result. Namely, we consider the $q$-deformations (or ``$q$-analogues'') of real numbers, which
have been recently introduced in~\cite{SVRat,SVRe} and studied further in \cite{CSS,SVJ,LMG}. 
They have several nice properties and connections, including theory of Conway-Coxeter friezes and knot invariants.  

For a rational $x=\frac{r}{s}>0$ 
the $q$-deformation is a rational function
\begin{equation}
\label{RS1}
\left[\frac{r}{s}\right]_{q}=\frac{\Rc(q)}{\Sc(q)},
\end{equation}
where~$\Rc(q)$ and~$\Sc(q)$ are coprime polynomials with positive integer coefficients
both depending on~$r$ and~$s$.
When $x\geq1$ is an irrational, the $q$-deformation of~$x$ is defined as a power series in~$q$:
\begin{equation}
\label{TayEq}
\left[x\right]_q=1+\varkappa_1q+\varkappa_2q^2+\varkappa_3q^3+\cdots
\end{equation}
with coefficients $\varkappa_{k}\in\Z$.
To obtain the series~\eqref{TayEq}, 
one chooses an arbitrary sequence of rationals~$(x_i)_{i\in\Z}$ converging to~$x$.
It turns out that  the Taylor series of the rational functions~$\left[x_i\right]_q$
stabilize, as $i$ grows; see~\cite{SVRe}.
Moreover, the stabilized series depends only on~$x$ (and not on the approximating sequence of rationals).
The series~$\left[x\right]_q$ is defined as the stabilization of
the Taylor series of~$\left[x_i\right]_q$.
It is unknown, in general, how to characterize the class of power series that
represent $q$-deformed real numbers.
One of the goals of this paper is to demonstrate that series arising in this context are 
not arbitrary.

We investigate analytic properties of $q$-deformed real numbers.
Considering the parameter of deformation~$q$ as a {\it complex variable}, 
$q\in\C$, we study the radius of convergence, $R(x)$,
of the series~\eqref{TayEq}. 
 In particular, we will show that the radius of the golden ratio $\varphi=\frac{1+\sqrt{5}}{2}$ is
$$
R_\star:=R(\varphi)=\frac{3-\sqrt{5}}{2}=0.381966...
$$
We have the following conjecture, which can be considered as a possible $q$-analogue of the Hurwitz claim.

\begin{conj}
\label{MainConj}
 For every real $x>0$ the radius of convergence $R(x)$ of the series~$\left[x\right]_q$
satisfies the inequality
$
R(x)\geq R_\star
$
and the equality holding only for $x$ which are $\PSL(2,\Z)$-equivalent to~$\varphi$.
\end{conj}

 Conjecture~\ref{MainConj} was checked by the extensive computer experiments.
 In this paper we get the following general results towards the proof of the conjecture.

\begin{thm}
\label{NewThm}
For every rational number $x>0$ the radius of convergence $R(x)$ of the series~$\left[x\right]_q$
has the following lower bound
$$
R(x) > 3-2\sqrt{2}=0.171572...
$$
\end{thm}

\begin{thm}
\label{NewThm2}
If all coefficients of the Hirzebruch-Jung 
continued fraction expansion of a rational number $x=\llbracket{}c_1,c_2,c_3,\ldots,c_N\rrbracket{}$
are larger than 3,
then the radius of convergence of~$\left[x\right]_q$ is greater than ~$R_\star.$
\end{thm}

Our main analytical tool is the classical Rouch\'e theorem (see~\cite{Conway,Tit}). Note that in the rational case the radius of convergence of the Taylor series can also be understood
as the radius of the maximal disc around~$0$ that contains no zeros of the polynomial~$S(q)$ in
the denominator of~\eqref{RS1}. 
 We refer also to the recent work ~\cite{Ren}, where our technique was used to prove the conjecture for a very particular
set of numbers (called metallic).

Conjecture~\ref{MainConj} and Theorem~\ref{NewThm} give a restriction for the series~\eqref{TayEq}
that can appear as a $q$-deformed real number.
However, it is quite clear that the radius of convergence is not the only condition.
It would be interesting to obtain more information about this class of power series.

Note that the inverse $\rho=R^{-1}$ of the radius of convergence of power series (\ref{TayEq}) can be given by the standard formula (see e.g. \cite{Conway})
$
\rho(x)=\limsup_{n\to\infty} |\varkappa_n|^{\frac{1}{n}}
$
and thus describes the growth of the coefficients~$\varkappa_n.$
Conjecture~\ref{MainConj} can be reformulated in the form, more similar to Hurwitz, as the inequality
$$
\rho(x)\leq \rho(\varphi)=\frac{3+\sqrt{5}}{2}=\varphi^2
$$
holding for every real $x$, and thus means that~$\left[\varphi\right]_q$ 
has the fastest growth of the coefficients among all $q$-deformed real numbers.
In other words, the series~$\left[x\right]_q$ always ``converges better'' than~$\left[\varphi\right]_q$,
whenever~$x$ is not equivalent to~$\varphi$.
In the particular case where~$x=\frac{r}{s}$ is rational, Conjecture~\ref{MainConj} means that
the polynomial~$\Sc(q)$ in~\eqref{RS1} has no roots $q$ with $|q|<R_\star.$

We should mention that the analogy with the Diophantine analysis here cannot be extended to the celebrated Markov theorem \cite{Markov}, claiming that the set of all possible Markov constants $\mu=\mu(x)>\frac{1}{3}$ is discrete (see e.g. \cite{Aigner}). 
The situation might be more similar to the Lyapunov spectrum of Markov and Euclid trees, 
which fills the whole segment $[0, \ln \varphi]$, see \cite{SV}.
The paper is organized as follows.
In Section~\ref{ExpS}, we briefly recall the notion of $q$-rationals and $q$-irrationals.
Following~\cite{LMG}, we emphasize the role of the modular group~$\PSL(2,\Z)$.
We also give explicit formulas in terms of continued fractions.

In Section~\ref{GRS}, we consider the important example of $q$-deformed golden ratio.
Its approximation by the quotients of consecutive Fibonacci numbers leads to interesting ``Fibonacci polynomials''.
We prove that the roots of these polynomials belong to the annulus
$
R_0<\left|q\right|<R_0^{-1},
$
where $R_0=0.35320...$ is the real root of the cubic equation
$$
R^3+2R^2+2R-1=0.
$$
This provides the lower estimate $R>R_0$ for the radius of convergence of the corresponding $q$-convergents.

In Section~\ref{PellS}, we consider the $q$-version of the silver ratio ~$\left[\sqrt{2}\right]_q$
and the related ``Pell polynomials''.
We prove that the roots of these polynomials belong to the annulus
$
R_1<\left|q\right|<R_1^{-1}
$
where $R_1=0.43542...$ is the smallest positive root of the equation
$$
R^4-2R^3-2R+1=0.
$$

Finally, in Section~\ref{GeS} we prove our Theorem~\ref{NewThm} and Theorem~\ref{NewThm2}.

\section{Definition and explicit formulas for $q$-reals}\label{ExpS}

We start with recalling an axiomatic definition of $q$-deformed rational numbers.
For more equivalent definitions see~\cite{SVRat,SVJ}.

\subsection{An axiomatic definition}\label{AxDefSec}
Recall that every rational can be obtained from~$0$, by applying a sequence of the operations
$x\mapsto{}x+1$ and~$x\mapsto{}-\frac{1}{x}$.

The following two recurrences (see~\cite{LMG}) suffice for calculation of $q$-rationals.

\begin{defn}
\label{RecDefQq}
The $q$-deformation sends every rational number $x=\frac{r}{s}$ to a rational function in~$q$
$$
x\longmapsto\left[x\right]_q=\frac{\Rc(q)}{\Sc(q)},
$$
in such a way that the following two recurrent formulas are satisfied
\begin{equation}
\label{SLEq}
\left[x+1\right]_q=q\left[x\right]_q+1,
\qquad\qquad
\left[-\frac{1}{x}\right]_q=-\frac{1}{q\left[x\right]_{q}}.
\end{equation}
Recurrences~\eqref{SLEq} determine the $q$-deformation of a rational in a unique way, starting from the
``initial'' trivial deformation~$\left[0\right]_q=0$; see~\cite{SVJ}.
\end{defn}

\begin{ex}
\label{Q5}
The following examples are easy to obtain applying~\eqref{SLEq}.

(a) The $q$-deformation of integers corresponds to the standard formulas of Euler and Gauss.
For $n\in\N$, one has
$$
\left[n\right]_{q}=
1+q+q^2+\cdots+q^{n-1},
\qquad\qquad
\left[-n\right]_{q}=
-q^{-n}-q^{1-n}-q^{2-n}-\cdots-q^{-1}.
$$
Both cases can be written as~$\left[x\right]_{q}=\frac{1-q^x}{1-x}$, where~$x$ is integer.

(b)
Already the simplest example of~$\frac{1}{2}$ shows that the $q$-deformation is not the quotient
of the $q$-deformed integers appearing in numerator and denominator:
$$
\left[\frac{1}{2}\right]_{q}=
\frac{q}{1+q},
\qquad\qquad
\left[-\frac{1}{2}\right]_{q}=
-\frac{1}{q\left(1+q\right)}.
$$
The above expression~$\left[x\right]_{q}=\frac{1-q^x}{1-x}$ is no longer true.

(c)
The next examples
$$
\left[\frac{5}{2}\right]_{q}=
\frac{1+2q+q^{2}+q^{3}}{1+q},
\qquad\qquad
\left[\frac{5}{3}\right]_{q}=
\frac{1+q+2q^{2}+q^{3}}{1+q+q^{2}}
$$
illustrate the fact
that the numerator and denominator in~\eqref{RS1} depend simultaneously on~$r$ and~$s$.
Indeed, the ``quantized~$5$'' in the numerator depends on the denominator.
\end{ex}

\begin{rem}
\label{IrRem}
In the case of $q$-rationals, one also has the formula for the negation and inverse:
\begin{equation}
\label{qInvEq}
\left[-x\right]_q=-q^{-1}\left[x\right]_{q^{-1}},
\qquad\qquad
\left[\frac{1}{x}\right]_q=\frac{1}{\left[x\right]_{q^{-1}}}.
\end{equation}
\end{rem}

\subsection{The action and a central extension of $\PSL(2,\Z)$}
We briefly mention here a more conceptual way of understanding Definition~\ref{RecDefQq}.
This observation will not be used in the sequel, and will be part of a separate work.

Recurrences~\eqref{SLEq} can be reformulated as a statement that
``$q$-deformation commutes with $\PSL(2,\Z)$-action''.
Indeed, recurrences~\eqref{SLEq} define an
action of the modular group $\PSL(2,\Z)$ on $q$-deformed rationals.
This action is given by fractional-linear transformations and is generated by the matrices
\begin{equation}
\label{TqSq}
T_q=\begin{pmatrix}
q&1\\[4pt]
0&1
\end{pmatrix},
\quad
\quad
S_q=\begin{pmatrix}
0&-1\\[4pt]
q&0
\end{pmatrix}
\end{equation}
which are $q$-deformations of the standard generators,~$T,S$, of~$\PSL(2,\Z)$
corresponding to~$q=1$ in~\eqref{TqSq}.

The relations $S^2=\Id$ and $(TS)^3=\Id$, defining~$\PSL(2,\Z)$, become
$$
S_q^2=q\,\Id,
\qquad\qquad
\left(T_qS_q\right)^3=q^3\Id,
$$
The matrices~$T_q,S_q$ generate an interesting non-trivial {\it central extension} of~$\PSL(2,\Z)$,
which is different from the braid group~$B_3$, yet contains it as a subgroup.
The centre, $\{q^n\Id\;|\;n\in\Z\}$, of the extended modular group 
acts trivially of $q$-deformed rationals,
so that one still has the well-defined action of~$\PSL(2,\Z)$.

\begin{rem}
(a)
Let us stress on the fact that emergence of a central extension of the symmetry group
is a usual situation in geometry; see~\cite{Kir}.
However, the centre of the extended group usually acts trivially on quantized objects.

(b)
The matrix~$S_q$ arose in the context of quantum groups; see~\cite{BK},
while $T_q$ can be viewed as a ``standard'' matrix connected to quantum integers.
However, we did not find in the literature simultaneous appearance of~$T_q$ and~$S_q$.
\end{rem}

\subsection{$q$-deformed continued fractions}{\label{qCFSec}

Every rational number~$\frac{r}{s}>0$ , where $r,s\in\Z_{>0}$ are coprime,
has a standard finite continued fraction expansion
$$
\frac{r}{s}
\quad=\quad
a_1 + \cfrac{1}{a_2 
          + \cfrac{1}{\ddots +\cfrac{1}{a_{2m}} } },
          $$
where $a_i\geq1$ (except for $a_1\geq0$).
It is usually denoted by  $\frac{r}{s}=[a_1,\ldots,a_{2m}]$.
Note that, choosing an even number of coefficients, one
removes the ambiguity $[a_1,\ldots,a_{n},1]=[a_1,\ldots,a_{n}+1]$ and makes the expansion unique.

There is also a unique continued fraction expansion 
with minus signs, often called the Hirzebruch-Jung continued fraction:
$$
\frac{r}{s}
 \quad =\quad
c_1 - \cfrac{1}{c_2 
          - \cfrac{1}{\ddots - \cfrac{1}{c_N} } } ,
$$
where $c_j\geq2$ (except for $c_1\geq1$).
The notation used by Hirzebruch is $\frac{r}{s}=\llbracket{}c_1,\ldots,c_N\rrbracket{}$.
The coefficients $a_i$ and $c_j$ of the above expansions
are connected by the Hirzebruch
formula; see, e.g.,~\cite{Hir,MGO} and Section~\ref{StatGTSec}.

\begin{defn}
\label{DefRat}
(a)
The $q$-deformed regular continued fraction is defined by
\begin{equation}
\label{qa}
[a_{1}, \ldots, a_{2m}]_{q}:=
[a_1]_{q} + \cfrac{q^{a_{1}}}{[a_2]_{q^{-1}} 
          + \cfrac{q^{-a_{2}}}{[a_{3}]_{q} 
          +\cfrac{q^{a_{3}}}{[a_{4}]_{q^{-1}}
          + \cfrac{q^{-a_{4}}}{
        \cfrac{\ddots}{[a_{2m-1}]_q+\cfrac{q^{a_{2m-1}}}{[a_{2m}]_{q^{-1}}}}}
          } }} 
\end{equation}
where~$[a]_q$ is the Euler $q$-integer.

(b)
The $q$-deformed Hirzebruch-Jung continued fraction is
\begin{equation}
\label{qc}
\llbracket{}c_1,\ldots,c_N\rrbracket_{q}:=
[c_1]_{q} - \cfrac{q^{c_{1}-1}}{[c_2]_{q} 
          - \cfrac{q^{c_{2}-1}}{\ddots \cfrac{\ddots}{[c_{N-1}]_{q}- \cfrac{q^{c_{N-1}-1}}{[c_N]_{q}} } }} 
\end{equation}
\end{defn}

For every rational $\frac{r}{s}$ written in two different ways: 
$$
\frac{r}{s}=[a_1,\ldots,a_{2m}]=\llbracket{}c_1,\ldots,c_N\rrbracket{},
$$
the rational functions~\eqref{qa} and~\eqref{qc} coincide,
and also coincide with the rational function~$\left[\frac{r}{s}\right]_q$
provided by Definition~\ref{RecDefQq}; see~\cite{SVRat}.

\subsection{Stabilization phenomenon, $q$-irrationals}\label{IrrSec}

The notion of $q$-deformed rational
was extended to irrational numbers in~\cite{SVRe}.
Let $x\geq1$ be an irrational number, and choose a sequence of rationals~$(x_n)_{n\geq1}$, converging to~$x$.
Consider the corresponding sequence of $q$-rationals
$\left[x_1\right]_q,\left[x_2\right]_q,\ldots$
It turns out that this sequence of rational functions also converge, but in the sense of formal power series.

Consider the Taylor expansions at~$q=0$ of the rational functions $\left[x_n\right]_q$, that,
abusing notation, we also denote by~$\left[x_n\right]_q$:
$$
\left[x_n\right]_q=\sum_{k\geq0}\varkappa_{n,k}\,q^k.
$$
One has the following stabilization property.

\begin{thm}[\cite{SVRe}]
\label{ConvThm}
(i) For every~$k\geq0$, the coefficients~$\varkappa_{n,k}$ of the Taylor series 
of the functions~$\left[x_n\right]_q$ stabilize, as~$n$ grows.

(ii) The limit coefficients,~$\varkappa_k:=\lim_{n\to\infty}\varkappa_{n,k}$,
do not depend on the choice of the sequence of rationals, but only on the irrational number~$x$.
\end{thm}

The $q$-deformation, $\left[x\right]_q$, is the limit power series in~$q$,
which has the form~\eqref{TayEq} (i.e., the coefficient of degree~$0$ is equal to~$1$).
Furthermore, the recurrence~\eqref{SLEq} also holds for the series~$\left[x\right]_q$,
and allows one to extend the $q$-deformation to the case of~$x<1$.
In particular, if~$x$ is negative, the resulting series in~$q$ is a {\it Laurent series} (with integer coefficients):
$$
\left[x\right]_q=-q^{-N}+\varkappa_{1-N}\,q^{1-N}+\varkappa_{2-N}\,q^{2-N}+\cdots
$$
where $N\in\Z_{>0}$ such that $-N\leq{}x<1-N$.

\begin{rem}
Let us mention that the stabilization phenomenon fails when a sequence of rationals~$(x_n)_{n\geq1}$, converges to
another rational, cf.~\cite{SVRe,Asi}.
\end{rem}

\section{Convergence radius of $q$-golden ratio and roots of the Fibonacci polynomials}\label{GRS}

The simplest example of $q$-irrational is the $q$-deformation of the celebrated golden ratio,
$\varphi=\frac{1+\sqrt{5}}{2}$.
The series~$\left[\varphi\right]_q$ is obtained as stabilized Taylor series of the $q$-deformed quotients of the
consecutive Fibonacci numbers $\left[\frac{F_{n+1}}{F_{n}}\right]_q$ (see~\cite{SVRat,SVRe}).
We call the polynomials in the numerator and the denominator of these rational functions 
the {\it Fibonacci polynomials}.

We prove that the radius of convergence of~$\left[\varphi\right]_q$ is~$R_\star=\frac{3-\sqrt{5}}{2}$, 
and  that all roots of the Fibonacci polynomials
belong to the annulus bounded by the circles 
with the radius~$R_\star=\frac{3-\sqrt{5}}{2}$ and~$R_\star^{-1}=\frac{3+\sqrt{5}}{2}$.
This result is in an accordance with Conjecture~\ref{MainConj}.

\subsection{The $q$-deformed golden ratio}

The $q$-deformation of the golden ratio,
$\varphi=\frac{1+\sqrt{5}}{2}$,
was considered in~\cite{SVRe}.
The series~$\left[\varphi\right]_q$ can be written
as an infinite continued fraction:
\begin{equation}
\label{GCF}
\left[\varphi\right]_q=
1 + \cfrac{q^{2}}{q
          + \cfrac{1}{1 
          +\cfrac{q^{2}}{q
          + \cfrac{1}{\ddots
       }}}} 
             \; =\;
1 + \cfrac{1}{q^{-1}
          + \cfrac{1}{q^2
          +\cfrac{1}{q^{-3}
          + \cfrac{1}{\ddots
       }}}} 
\end{equation}

The series starts as follows
$$
\begin{array}{rcl}
\left[\varphi\right]_q&=&
1 + q^2 - q^3 + 2 q^4 - 4 q^5 + 8 q^6 - 17 q^7 + 37 q^8 - 82 q^9 + 185 q^{10} \\[6pt]
&&- 423 q^{11} + 978 q^{12}-2283q^{13}+ 5373q^{14}-12735q^{15}+30372q^{16}\\[6pt]
&&-72832q^{17}+175502q^{18}-424748q^{19}+1032004q^{20} \cdots
\end{array}
$$
The sequence of coefficients in~$\left[\varphi\right]_q$ coincides (up to the alternating sign)
with the remarkable sequence A004148 of \cite{OEIS} called the ``generalized Catalan numbers''.

The series~$\left[\varphi\right]_q$ is a solution of the functional equation
\begin{equation}
\label{GREq}
q\,X^2=
\left(q^2+q-1 \right)X +1,
\end{equation}
which can be deduced from~\eqref{GCF}.

\begin{prop}
\label{GRProp}
The radius of convergence of the series~$\left[\varphi\right]_q$ is equal to~$R_\star$.
\end{prop}

\begin{proof}
It follows from~\eqref{GREq}, that
the generating function of $\left[\varphi\right]_q$ can therefore be written in radicals:
\begin{equation}
\label{GGF}
\mathrm{GF}_{\left[\varphi\right]_q}=
\frac{q^2+q-1+\sqrt{(q^2 + 3q + 1)(q^2 - q + 1)}}{2q},
\end{equation}
and the series~$\left[\varphi\right]_q$ is the Taylor expansion of~$\mathrm{GF}_{\left[\varphi\right]_q}$ at $q=0$.
The number~$R_\star$ is the modulus of the smallest (i.e., closest to~$0$) root of the polynomials under the radical in~\eqref{GGF}.
Indeed, 
$$
q^2 + 3q + 1=
\left(q+R_\star\right)\left(q+R_\star^{-1}\right),
$$
and therefore the Taylor series of~\eqref{GGF} converges for~$|q|<R_\star$.
\end{proof}

\begin{rem}
Observe that formula~\eqref{GCF} has a certain similarity 
with the celebrated Rogers-Ramanujan continued fraction
$$
R(q)=
1 + \cfrac{1}{1
          + \cfrac{q}{1 
          +\cfrac{q^{2}}{1
          + \cfrac{q^3}{\ddots
       }}}} 
$$
but the deformation~\eqref{GCF} has very different properties.
\end{rem}

\subsection{The Fibonacci polynomials}

The most natural choice of a sequence of rationals converging to~$\varphi$ is related to a
quite remarkable and well-known sequence of polynomials.

Let~$F_n$ be the~$n^{\hbox{th}}$ Fibonacci number,
the sequence of rationals $\frac{F_{n+1}}{F_{n}}$ converges to~$\varphi$.
Quantizing this sequence, one obtains a sequence of rational functions
\begin{equation}
\label{FibPolDef}
\left[\frac{F_{n+1}}{F_{n}}\right]_q=:
\frac{\tilde\F_{n+1}(q)}{\F_n(q)}.
\end{equation}
The polynomials~$\tilde\F_{n+1}(q)$ and~$\F_n(q)$
in the numerator and denominator of~\eqref{FibPolDef} are $q$-deformations of the Fibonacci numbers,
considered in~\cite{SVRat}.

Both sequences of polynomials~$\F_n(q)$ and~$\tilde\F_{n}(q)$ are of degree~$n-2$ (for~$n\geq2$)
and are mirror of each other:
$$
\tilde\F_n(q)=q^{n-2}\F_n(q^{-1}).
$$

The polynomials~$\F_n(q)$ can be calculated recursively.
It will be convenient to separate the sequence of polynomials~$\F_n(q)$ into two subsequences,
with even~$n$ and odd~$n$.
Both of these sequences satisfy the same recurrence,
which is a $q$-analogue of the classical recurrence
$$
F_{n+2}= 3\,F_n-F_{n-2}
$$ 
for the Fibonacci numbers.

\begin{prop}
\label{FiboP}
The polynomials~$\F_n(q)$ in the denominator of~\eqref{FibPolDef}
are determined by the recurrence
\begin{equation}
\label{FibPolRec}
\F_{n+2}(q)= [3]_q\,\F_n(q)-q^2\F_{n-2}(q),
\end{equation}
where $[3]_q=1+q+q^2$, and the initial conditions
$$
\left(\F_0(q)=1,\;
\F_2(q)=1+q\right)
\qquad\hbox{and}\qquad
\left(\F_1(q)=1,\;
\F_3(q)=1+q+q^2\right).
$$
\end{prop}

\begin{proof}
It was shown in~\cite{SVRat} that
$$
\begin{array}{rcl}
\F_{2\ell+1} &=& q\F_{2\ell}+\F_{2\ell-1},\\[4pt]
\F_{2\ell+2} &=& \F_{2\ell+1}+q^2\F_{2\ell},
\end{array}
$$
Recurrence~\eqref{FibPolRec} follows readily.
\end{proof}

The coefficients of the polynomials~$\F_n(q)$ and~$\tilde\F_{n}(q)$, with~$n\geq1$, form the triangles 
$$
\begin{array}{rcccccccc}
1\\
1&1\\
1&1&1\\
1&2&1&1\\
1&2&2&2&1\\
1&3&3&3&2&1\\
1&3&4&5&4&3&1\\
\cdots
\end{array}
\qquad\qquad
\begin{array}{rcccccccc}
&&&&&&1\\
&&&&&1&1\\
&&&&1&1&1\\
&&&1&1&2&1\\
&&1&2&2&2&1\\
&1&2&3&3&3&1\\
1&3&4&5&4&3&1\\
&&&&&\cdots
\end{array}
$$
known as Sequences~A123245 and~A079487 of OEIS~\cite{OEIS}, respectively.
\begin{ex}
\label{FibpPEx}
One has
$$
\begin{array}{rcl}
\left[\frac{5}{3}\right]_q&=&\displaystyle\frac{1+q+2q^2+q^3}{1+q+q^2},
\\[12pt]
\left[\frac{8}{5}\right]_q&=&\displaystyle\frac{1+2q+2q^2+2q^3+q^4}{1+2q+q^2+q^3},
\\[12pt]
\left[\frac{13}{8}\right]_q&=&\displaystyle\frac{1+2q+3q^2+3q^3+3q^4+q^5}{1+2q+2q^2+2q^3+q^4},
\\[12pt]
\left[\frac{21}{13}\right]_q&=&\displaystyle\frac{1+3q+4q^2+5q^3+4q^4+3q^5+q^6}{1+3q+3q^2+3q^3+2q^4+q^5},
\\
\ldots&\ldots&\ldots
\end{array}
$$
\end{ex}

\begin{rem}
The polynomials~$\F_n(q)$
with odd~$n$ are specializations of $3$-parameter family of polynomials
considered in~\cite{Var} (see Remark 8.4.).
\end{rem}

\subsection{Roots of the Fibonacci polynomials}
Our next goal is to obtain the bounds for the absolute values of roots of the Fibonacci polynomials.
For this we will use the following classical Rouch\'e theorem (see~\cite{Conway,Tit}).

\begin{theo}[Rouch\'e]
Let $f$ and~$g$ be two functions of one complex variable~$q$, holomorphic inside a disc~$D$
and continuous on the bound~$\partial{}D$.
Suppose that $\left| f(q)\right|>\left| g(q)\right|$ on~$\partial{}D$, then $f$ and~$f+g$ have the same number of zeros inside~$D$.
\end{theo}

Let $R_0=0.35320...$ be the positive root of the equation
\begin{equation}
\label{R0}
R^3+2R^2+2R-1=0.
\end{equation}

Let~$D_0$ be the disc and $C_0$ the circle with radius~$R_0$:
$$
D_0=\left\{q\in\C\;,|q|\leq R_0\right\},
\qquad\qquad
C_0=\left\{q\in\C\;,|q|=R_0\right\}.
$$

\begin{thm}
\label{FiboT}
For every~$n\in\N$ the roots of the Fibonacci polynomials~$\F_n(q)$ and~$\tilde\F_{n}(q)$
belong to the annulus
$$
R_0<\left|q\right|<R_0^{-1}.
$$
\end{thm}

\begin{proof}
We need the following useful lemma.

\begin{lem}
\label{MinLem}
For every positive integer~$n$ one  has
\begin{equation}
\label{EstLem}
\left|{\left[n\right]_q}\right|\geq\frac{1-R^{n}}{{1+R}},
\end{equation}
on a circle $C_R=\{q\in \mathbb C: |q|=R\}$ of any radius~$R$.
\end{lem}

\begin{proof}
Indeed, the modulus of the polynomial
$
\left[c\right]_q=
\frac{1-q^n}{1-q}
$
restricted to $C_R$ is of the form
$$
\left|\left[n\right]_q\right|_{C_R}^2=
\frac{(1-R^ne^{in\theta})(1-R^ne^{-in\theta})}{(1-Re^{i\theta})(1-Re^{-i\theta})}=
\frac{1-2R^n\cos(2n\theta)+R^{2n}}{1-2R\cos(\theta)+R^2},
$$
where $\theta$ is the standard parameter on the circle.
Observe that the numerator of this fraction is greater 
than the numerator of the square of the right-hand-side of~\eqref{EstLem},
while the denominator is smaller.
The lemma follows.
\end{proof}

Recurrence~\eqref{FibPolRec} can be rewritten as follows
\begin{equation}
\label{FibFracRec}
\frac{\F_{n+2}(q)}{\F_n(q)}= [3]_q-q^2\,\frac{\F_{n-2}(q)}{\F_n(q)}.
\end{equation}

\begin{lem}
\label{SecL}
Assume that $\left|\frac{\F_{n-2}(q)}{\F_n(q)}\right|_{C_0}<\frac{1}{R_0}$, 
then $\left|\frac{\F_{n}(q)}{\F_{n+2}(q)}\right|_{C_0}<\frac{1}{R_0}$.
\end{lem}

\begin{proof}
Using the recurrence~\eqref{FibFracRec}, one has
$$
\left|\frac{\F_{n+2}(q)}{\F_n(q)}\right|_{C_{R_0}}\geq
\left|\left[3\right]_q\right|_{C_0}-R_0^{2}\left|\frac{\F_{n-2}(q)}{\F_n(q)}\right|_{C_0}.
$$
Lemma \ref{MinLem} and the assumption imply that
$$
\left|\frac{\F_{n+2}(q)}{\F_n(q)}\right|_{C_0}>
\frac{1-R_0^{3}}{{1+R_0}}-R_0.
$$
Since due to the definition of $R_0$ we have
$$
\frac{1-R_0^{3}}{{1+R_0}}-R_0=R_0,
$$
this implies the claim.
\end{proof}

Observe that, for $n=2$ we have $\frac{\F_{0}(q)}{\F_2(q)}=\frac{1}{\left[2\right]_q}$,
and $\frac{\F_{1}(q)}{\F_3(q)}=\frac{1}{\left[3\right]_q}$.
Using, once again, Lemma~\ref{MinLem}, we have $\left|\frac{\F_{0}(q)}{\F_2(q)}\right|_{C_0}<\frac{1}{R_0}$
and $\left|\frac{\F_{1}(q)}{\F_3(q)}\right|_{C_0}<\frac{1}{R_0}$.
From Lemma~\ref{SecL} we conclude by induction that $\left|\frac{\F_{n-2}(q)}{\F_n(q)}\right|_{C_0}<\frac{1}{R_0}$ for all $n\geq2.$

In particular, we have that 
\begin{equation}
\label{3C0}
\left|[3]_q\right|_{C_0}\geq \frac{1-R_0^{3}}{{1+R_0}}>R_0>R_0^2\left|\frac{\F_{n-2}(q)}{\F_n(q)}\right|_{C_0}.
\end{equation}


Let us assume now that $\F_{n-2}(q)$ and $\F_{n}(q)$ have no zeros inside $D_0$. Since the cyclotomic polynomial $[3]_q=q^2+q+1$ has no roots inside~$D_0$, then from \eqref{FibFracRec} and (\ref{3C0}) by
 the Rouch\'e theorem, 
we conclude that~$\F_{n+2}(q)$ also has no roots  inside~$D_0$. 
The induction completes the proof of the statement that the polynomials $\F_{n}(q)$ have no zeros inside $D_0$.

%
%


The proof of the
same statement for the polynomials~$\tilde\F_n(q)$ is similar, since~$\tilde\F_n(q)$ also satisfy~\eqref{FibPolRec}.
This follows from the recurrent formulas
$$
\begin{array}{rcl}
\tilde\F_{2\ell+1} &=& \tilde\F_{2\ell}+q^2\tilde\F_{2\ell-1},\\[4pt]
\tilde\F_{2\ell+2} &=& q\tilde\F_{2\ell+1}+\tilde\F_{2\ell},
\end{array}
$$
proved in~\cite{SVRat}.

Finally, since the polynomials~$\F_n(q)$ and~$\tilde\F_{n}(q)$ are mirror of each other,
this implies that these polynomials have no roots outside the disc with the radius $R_0^{-1}$.
Theorem~\ref{FiboT} is proved.
\end{proof}

The computer calculations shows that the roots of the Fibonacci polynomials
actually lie in the smaller annulus
$$
R_\star<\left|q\right|<R_\star^{-1},
$$
but we cannot prove this yet. {
This was claimed in the earlier preprint version of this work (and repeated in \cite{Ren}), but our proof was incorrect.

\section{$\left[\sqrt{2}\right]_q$ and roots of the Pell polynomials}\label{PellS}

The second example we consider is~$\sqrt{2}$, known also as silver ratio.
The slight modification,~$\sqrt{2}+1$, has some tiny advantages, so we will often work with it;
note that~$\sqrt{2}+1$ is approximated by the quotients of consecutive {\it Pell numbers}
$\frac{P_{n+1}}{P_n}=\underbrace{\left[2, 2, \ldots,2\right]}_n$.

We show that the convergence radius of the series~$\left[\sqrt{2}\right]_q$ and~$\left[\sqrt{2}+1\right]_q$ is equal to $R_{\sqrt{2}}:=R(\sqrt{2})$,
\begin{equation}
\label{Rad1}
R_{\sqrt{2}}=
\frac{1 + \sqrt{2} - \sqrt{2\sqrt{2} - 1}}{2}\quad\approx0.53101,
\end{equation}
and that the roots of the polynomials in the numerators and
denominators of~$\left[\frac{P_{n+1}}{P_n}\right]_q$
belong to the annulus bounded by the circles with radius~$R_1$ and~$R_1^{-1}$, with $R_1\approx0.43542$.

\subsection{The series~$\left[\sqrt{2}\right]_q$ and~$\left[\sqrt{2}+1\right]_q$}
These series are related by
$\left[\sqrt{2}+1\right]_q=q\left[\sqrt{2}\right]_q+1$ (cf.~\eqref{SLEq}) and obviously have the same convergence radius;
we prefer to perform the calculations for~$\left[\sqrt{2}+1\right]_q$.

The $q$-deformation~$\left[\sqrt{2}+1\right]_q$ is given by the infinite $2$-periodic continued fraction
\begin{equation}
\label{RRSEq}
\left[\sqrt{2}+1\right]_q=
1 +q+ \cfrac{q^{4}}{q+q^2
          + \cfrac{1}{1 +q
          +\cfrac{q^{4}}{q+q^2
          + \cfrac{1}{\ddots
       }}}} 
\end{equation}
see~\cite{SVRe}.
This is the $q$-deformed classical continued fraction expansion~$\sqrt{2}+1=\left[2,2,2,2,\ldots\right]$.

The series $\left[\sqrt{2}+1\right]_q$ satisfies the following functional equation:
\begin{equation}
\label{SREq}
qX^2-
\left(q^3+2q-1 \right)X -1 =0,
\end{equation}
readily obtained from~\eqref{RRSEq}, and can be calculated from it recursively:
$$
\begin{array}{rcl}
\left[\sqrt{2}+1\right]_q&=&
1+q+q^4-2q^6+q^7+4q^8-5q^9-7q^{10}+ 18q^{11}+ 7q^{12}-55q^{13}+ 18q^{14}\\[4pt]
&&+ 146q^{15}- 155q^{16} - 322q^{17}+692q^{18}+ 476q^{19}- 2446q^{20}+ 307q^{21}\\[4pt]
&&+ 7322q^{22}- 6276q^{23}- 18277q^{24}+ 33061q^{25}+ 33376q^{26}- 129238q^{27}- 10899q^{28}\cdots
\end{array}
$$
see Sequence A337589 of~\cite{OEIS} for the coefficients of this series.

\begin{prop}
\label{GRProp}
The radius of convergence of the series~$\left[\sqrt{2}\right]_q$ and~$\left[\sqrt{2}+1\right]_q$ is equal to~$R_{\sqrt{2}}$.
\end{prop}

\begin{proof}
The generating function of the series can be deduced from~\eqref{SREq}:
$$
\mathrm{GF}_{\left[\sqrt{2}+1\right]_q}=
\frac{q^3+2q-1+\sqrt{(q^4+q^3+4q^2+q+1)(q^2-q+1)}}{2q}.
$$
The radius~\eqref{Rad1} is equal to the modulus of the root of the polynomial
$q^4+q^3+4q^2+q+1$ closest to zero.
\end{proof}

\subsection{The Pell polynomials}
The irrational~$\left[\sqrt{2}+1\right]_q$ can be approximated by the quotient of the consecutive Pell numbers:
$\frac{P_{n+1}}{P_n}=
\underbrace{\left[2, 2, \ldots,2\right]}_n$.
We define the {\it Pell polynomials} via
$$
\left[\frac{P_{n+1}}{P_{n}}\right]_q=:
\frac{\tilde\Pc_{n+1}(q)}{\Pc_n(q)}.
$$
The polynomials $\tilde\Pc_{n}(q)$ and $\Pc_n(q)$ are of degree~$2n-3$, and, similarly to the Fibonacci polynomials,
are the mirrors of each other: $q^{2n-3}\tilde\Pc_{n}(q^{-1})=\Pc_n(q)$.

\begin{prop}
\label{PeP}
The polynomials~$\Pc_n(q)$  are determined by the recurrence
\begin{equation}
\label{PePolRec}
\Pc_{n+2} = {4\choose 2}_q\Pc_{n} - q^4\Pc_{n-2},
\end{equation}
where ${4\choose 2}_q=1+q+2q^2+q^3+q^4$ is the Gaussian $q$-binomial, and the initial conditions
$$
\left(\Pc_0(q)=0,\;
\Pc_2(q)=1+q\right)
\qquad\hbox{and}\qquad
\left(\Pc_{-1}(q)=\Pc_1(q)=1\right).
$$
\end{prop}

\begin{proof}
Recurrence~\eqref{PePolRec} follows from the formulas
$$
\begin{array}{rcl}
\Pc_{2\ell+1} &=& \left(q+q^2\right)\Pc_{2\ell}+\Pc_{2\ell-1},\\[4pt]
\Pc_{2\ell+2} &=& \left(1+q\right)\Pc_{2\ell+1}+q^4\,\Pc_{2\ell},
\end{array}
$$
proved in~\cite{SVRat}.
\end{proof}

The coefficients of~$\Pc_{n}(q)$ with~$n\geq1$ form a triangular sequence
$$
\begin{array}{rccccccccccc}
1\\
1&1\\
1&1&2&1\\
1&2&3&3&2&1\\
1&2&5&6&6&5&3&1\\
1&3&7&11&13&13&11&7&3&1\\
1 &  3  &9&  16&  24&  29&  29&  25&  18&   10&   4&   1\\
\cdots
\end{array}
$$
(see Sequence A323670 of~\cite{OEIS}).

\begin{ex}
\label{FibpPEx}
One has
$$
\begin{array}{rcl}
\left[\frac{5}{2}\right]_q&=&\displaystyle\frac{1+2q+q^2+q^3}{1+q+q^2},
\\[12pt]
\left[\frac{12}{5}\right]_q&=&\displaystyle\frac{1+2q+3q^2+3q^3+2q^4+q^5}{1+q+2q^2+q^3},
\\[12pt]
\left[\frac{29}{12}\right]_q&=&\displaystyle\frac{1+3q+5q^2+6q^3+6q^4+5q^5+2q^6+q^7}{1+2q+3q^2+3q^3+2q^4+q^5},
\\[12pt]
\left[\frac{70}{29}\right]_q&=&\displaystyle\frac{1+3q+7q^2+11q^3+13q^4+13q^5+11q^6+7q^7+3q^8+q^9}{1+2q+5q^2+6q^3+6q^4+5q^5+3q^6+q^7},
\\
\ldots&\ldots&\ldots
\end{array}
$$
\end{ex}

\subsection{Roots of the Pell polynomials}

Let $R_1=0.43542...$ be the smallest positive root of the equation
\begin{equation}
\label{R1}
R^4-2R^3-2R+1=0.
\end{equation}
Let~$D_1$ be the disc and $C_1$ the circle with radius~$R_1$:
$$
D_1=\left\{q\in\C\;,|q|\leq R_1\right\},
\qquad\qquad
C_1=\left\{q\in\C\;,|q|=R_1\right\}.
$$

\begin{thm}
\label{PeT}
The roots of the Pell polynomials ~$\Pc_n(q)$ and~$\tilde\Pc_{n}(q)$ belong to the annulus 
$$
R_1<\left|q\right|<R_1^{-1}.
$$
\end{thm}

\begin{proof}
The proof is similar to the Fibonacci case.
Recurrence~\eqref{PePolRec} leads to
\begin{equation}
\label{PeFracRec}
\frac{\Pc_{n+2}(q)}{\Pc_n(q)}= {4\choose 2}_q-q^4\,\frac{\Pc_{n-2}(q)}{\Pc_n(q)},
\end{equation}
where ${4\choose 2}_q=(q^2+1)(q^2+q+1).$
On the circle $C_R$ of radius $R$ we have $|1+q^2|\geq 1-R^2$
 and $|q^2+q+1|\geq \frac{1-R^3}{1+R}$ as before.
 Thus on $C_R$ we have  
 \begin{equation}
\label{EstLemPell}
\left|{4\choose 2}_q\right|\geq\frac{(1-R^{3})(1-R^2)}{1+R}=(1-R^3)(1-R).
\end{equation}

%
%
%

\begin{lem}
\label{PSecL}
Assume that $\left|\frac{\Pc_{n-2}(q)}{\Pc_n(q)}\right|_{C_1}<\frac{1}{R_1}$, 
then $\left|\frac{\Pc_{n}(q)}{\Pc_{n+2}(q)}\right|_{C_1}<\frac{1}{R_1}$.
\end{lem}

{\it Proof of Lemma~\ref{PSecL}}.
It follows from~\eqref{PeFracRec} and (\ref{EstLemPell}) that
$$
\left|\frac{\Pc_{n+2}(q)}{\Pc_{n}(q)}\right|_{C_1}
\geq\left|{4\choose 2}_q\right|_{C_1}-R_1^4\left|\frac{\Pc_{n-2}(q)}{\Pc_n(q)}\right|_{C_1}>(1-R_1^3)(1-R_1)-R_1^3=R_1,
$$
since $(1-R_1^3)(1-R_1)-R_1^3-R_1=R_1^4-2R_1^3-2R_1+1=0$ by choice of $R_1.$
\qed

From the lemma by induction we can conclude that $\left|\frac{\Pc_{n-2}(q)}{\Pc_n(q)}\right|_{C_0}<\frac{1}{R_1}$
and thus
\begin{equation}
\label{3C1}
\left|{4\choose 2}_q\right|_{C_1}> (1-R_1^3)(1-R_1)=R_1^3+R_1>R_1^4\left|\frac{\Pc_{n-2}(q)}{\Pc_n(q)}\right|_{C_1}.
\end{equation}

Assume that $\Pc_{n-2}(q)$ and $\Pc_{n}(q)$ have no zeros inside $C_1$. Since  ${4\choose 2}_q=(q^2+1)(q^2+q+1)$ has no roots inside~$D_1$, then due to (\ref{3C1}) we can apply Rouch\'e theorem  
to conclude that~$\Pc_{n+2}(q)$ also has no roots  inside~$D_1$. The induction completes the proof
of this statement for all $\Pc_{n}(q).$

The proof that the polynomials~$\tilde\Pc_n(q)$ have no
roots inside~$D_1$ is analogous since~$\tilde\Pc_n(q)$ also satisfy~\eqref{PePolRec}.
Theorem~\ref{PeT} then follows from the fact that~$\Pc_n(q)$ and~$\tilde\Pc_n(q)$ are mirrors of each other.
\end{proof}

Computer experiments show that for $n\leq45$,
the roots of the polynomials~$\Pc_n(q)$ and~$\tilde\Pc_{n}(q)$ belong to the smaller annulus 
$R_{\sqrt{2}}<\left|q_r\right|<R_{\sqrt{2}}^{-1},$
see Figure 1. 

\begin{figure}[h]
\label{TheFig}
\includegraphics[scale=0.35]{Root.pdf}
\caption{Roots of $\Pc_{10}(q)$: the smallest root modulus is $0.5668\ldots$, and the largest $1.8832\ldots$}
\end{figure}

 \begin{rem}
In the earlier preprint version of this work we claimed that this is true for all $n$, however the proof was incorrect.
Moreover, according to our computer calculations, the minimal roots of~$\Pc_{47}(q)$ 
have absolute values $0.52883...$, which is smaller than~$R_{\sqrt{2}}$,
while the maximal absolute value is $1.88796...$, which is greater than~$R_{\sqrt{2}}^{-1}$.
We are not sure whether these computer results are reliable, so this has to be studied further.
We also mention that Ren \cite{Ren} has proved similar claim for the convergents of metallic numbers $[m,m,m, \dots]_q$ with $m=3,4.$
\end{rem}

\section{Proof of Theorems~\ref{NewThm}
and~\ref{NewThm2}}\label{GeS}

\subsection{Proof of Theorem~\ref{NewThm}}

Let~$x$ be a rational, choose the rational approximations of~$x$ 
by the convergents of the Hirzebruch-Jung continued fraction:
$\frac{r_n}{s_n}=\llbracket{}c_1,c_2,c_3,\ldots,c_n\rrbracket{}$, see Section~\ref{qCFSec},
and consider the $q$-deformation
\begin{equation}
\label{RS}
\frac{\Rc_n(q)}{\Sc_n(q)}:=
\left[\frac{r_n}{s_n}\right]_q.
\end{equation}
Recall that the coefficients of the Hirzebruch-Jung continued fraction are at least~$2$.

Let $R_2=3-2\sqrt{2}=(\sqrt{2}-1)^2$ and 
 $C_2$ and $D_2$ be the circle and the disc with radius~$R_2$, respectively.
To prove that the radius $R(x)$ of convergence of $[x]_q$ is larger than $R_2$, it would be enough to prove that for all~$n$ the polynomial~$\Sc_n(q)$ in the denominator of~\eqref{RS}
has no roots inside $D_2$.

We use the same strategy as before based on the recurrence relations and Rouch\'e theorem.
Formula~\eqref{qc} implies that the polynomial~$\Sc_n(q)$ satisfy the recurrence
\begin{equation}
\label{qc1}
\Sc_{n+1}(q)=\left[c_{n+1}\right]_q\,\Sc_n(q)-q^{c_n-1}\,\Sc_{n-1}(q),
\end{equation}
with the initial values~$\Sc_{0}(q)=0$ and~$\Sc_{1}(q)=1$, that we rewrite as follows
\begin{equation}
\label{GREq}
\frac{\Sc_{n+1}(q)}{\Sc_n(q)}=\left[c_{n+1}\right]_q\,-\,q^{c_n-1}\,\frac{\Sc_{n-1}(q)}{\Sc_n(q)}.
\end{equation}

By the Rouch\'e theorem,
it suffices to prove that, for every~$n$, 
the polynomial $\left[c_{n+1}\right]_q=\frac{\;1-q^{c_{n+1}}}{{1-q}}$ 
dominates the second summand of the right-hand-side of~\eqref{GREq},
when restricted on the circle~$C_{2}$:
\begin{equation}
\label{GRMBisEq}
\left|{\left[c_{n+1}\right]_q}\right|_{C_{2}}
\;>\;
{R_2}^{\,c_n-1}\left|\frac{\Sc_{n-1}(q)}{\Sc_n(q)}\right|_{C_{2}}.
\end{equation}
Since for every positive integer~$c$ the polynomial~$\left[c\right]_q$ has no roots in~$D_{2}$, we will then argue by induction that
$\Sc_{n+1}(q)$ also has no roots in~$D_{2}$.

%

To prove~\eqref{GRMBisEq}, we will need the following inductive step.

\begin{lem}
\label{LastLem}
If $\left|\frac{\Sc_{n-1}(q)}{\Sc_n(q)}\right|_{C_{2}}<\frac{1}{{R_2}^{\frac{1}{2}}}$, 
then $\left|\frac{\Sc_{n}(q)}{\Sc_{n+1}(q)}\right|_{C_{2}}<\frac{1}{{R_2}^{\frac{1}{2}}}$.
\end{lem}

\begin{proof}
Using the recurrence~\eqref{GREq}, one has
$$
\left|\frac{\Sc_{n+1}(q)}{\Sc_n(q)}\right|_{C_{2}}\geq
\left|\left[c_{n+1}\right]_q\right|_{C_{2}}-{R_2}^{c_n-1}\left|\frac{\Sc_{n-1}(q)}{\Sc_n(q)}\right|_{C_{2}}.
$$
 Using Lemma \ref{MinLem} and inductive assumption we have
$$
\left|\frac{\Sc_{n+1}(q)}{\Sc_n(q)}\right|_{C_{2}}>
\frac{1-{R_2}^{c_{n+1}}}{{1+{R_2}}}-{R_2}^{c_n-\frac{3}{2}}.
$$
We need to show that
$$
\frac{1-{R_2}^{c_{n+1}}}{{1+{R_2}}}-{R_2}^{c_n-\frac{3}{2}}\geq {R_2}^\frac{1}{2},
$$
or equivalently
$$
1-{R_2}^{c_{n+1}}-{R_2}^{c_n-\frac{1}{2}}-{R_2}^{c_n-\frac{3}{2}}-{R_2}^\frac{3}{2}-{R_2}^\frac{1}{2}\geq 0.
$$
Since $c_i\geq2$ it is enough to check that for $R=R_2$
$$
1-R^{2}-2R^{\frac{3}{2}}-2R^{\frac{1}{2}}\geq 0.
$$
But $1-R^{2}-2R^{\frac{3}{2}}-2R^{\frac{1}{2}}=(1+R)(1-R-R^{\frac{1}{2}})=0$ when
$R^\frac{1}{2}=\sqrt{2}-1=R_2^\frac{1}{2}$.

\end{proof}

Lemmas~\ref{MinLem} and~\ref{LastLem} imply that the inequality~\eqref{GRMBisEq} is guaranteed if
$$
\frac{1-{R_2}^{c_{n+1}}}{{1+{R_2}}} \geq {R_2}^{c_n-\frac{3}{2}},
$$
which is equivalent to
$$
1-{R_2}^{c_{n+1}}-{R_2}^{c_n-\frac{1}{2}}-{R_2}^{c_n-\frac{3}{2}}\geq0.
$$
Using once again $c_i\geq2$ it is enough to check that
$
1-R_2^{2}-R_2^{\frac{3}{2}}-R_2^{\frac{1}{2}}\geq0,
$
which is obvious since $1-R_2^{2}-R_2^{\frac{3}{2}}-R_2^{\frac{1}{2}}=R_2^{\frac{3}{2}}+R_2^{\frac{1}{2}}.$
This completes the proof of Theorem~\ref{NewThm}.

\subsection{A special class of continued fractions}\label{StatGTSec}


Now we prove Theorem \ref{NewThm2}, which claims that if the coefficients of the Hirzebruch-Jung 
continued fraction expansion a rational number $x=\llbracket{}c_1,c_2,c_3,\ldots,c_N\rrbracket{}$
satisfies
\begin{equation}
\label{InEq}
c_i\geq4,
\end{equation}
then the radius of convergence of~$\left[x\right]_q$ is greater than ~$R_\star=\frac{3-\sqrt{5}}{2}.$

Let us reformulate the inequality~\eqref{InEq} in terms of the coefficients of the
regular continued fraction expansion~$x=[a_1,a_2,a_3\ldots]$.
Recall the formula~\cite{Hir} expressing the coefficients of the Hirzebruch-Jung continued fraction:
$$
x=\llbracket{}
a_1+1,\underbrace{2,\ldots,2}_{a_2-1},\,
a_3+2,\underbrace{2,\ldots,2}_{a_4-1},\,a_5+2,\ldots,
a_{2n-1}+2,\underbrace{2,\ldots,2}_{a_{2n}-1},
\ldots\rrbracket
$$
In other words, the coefficients with odd indices,~$a_{2m-1}$ become $a_{2m-1}+2$
(except for~$a_1$ that produces $a_1+1$ in the Hirzebruch-Jung 
continued fraction), and coefficients with even indices,~$a_{2m}$ produce an $(a_{2m}-1)$-tuple of~$2$'s.
Inequality~\eqref{InEq} becomes:
$$
\left\{
\begin{array}{rcl}
a_{2m-1} & \geq & 2,\\[4pt]
a_{2m} & = & 1,
\end{array}
\right.
$$
starting from some~$N$.

\subsection{Proof of Theorem~\ref{NewThm2} }

The proof goes along the same lines as that of Theorem~\ref{NewThm}.
Let $C_\star$ be the circle with radius~$R_\star$.
It suffices to prove that for every~$n$ the polynomial~$\Sc_n(q)$ in~\eqref{RS} has no roots inside ~$C_\star$.

We need to prove that for every~$n$ 
the polynomial $\left[c_{n+1}\right]_q=\frac{1-q^{c_{n+1}}}{{1-q}}$ 
dominates the second summand of the right-hand-side of the recursion~\eqref{GREq},
when restricted on the circle~$C_\star$:
\begin{equation}
\label{GRMEq}
\left|{\left[c_{n+1}\right]_q}\right|_{C_\star}
\;>\;
R_\star^{\,c_n-1}\left|\frac{\Sc_{n-1}(q)}{\Sc_n(q)}\right|_{C_\star}.
\end{equation}

We argue by induction and will need the following.

\begin{lem}
\label{LastLem2}
If~$\left|\frac{\Sc_{n-1}(q)}{\Sc_n(q)}\right|_{C_\star}<\frac{1}{R_\star}$, then 
$\left|\frac{\Sc_{n}(q)}{\Sc_{n+1}(q)}\right|_{C_\star}<\frac{1}{R_\star}$.
\end{lem}

\begin{proof}
Using \eqref{GREq} and lemma \ref{MinLem} we have
$$
\left|\frac{\Sc_{n+1}(q)}{\Sc_{n}(q)}\right|_{C_\star}>\frac{1-R_\star^{c_{n+1}}}{1+R_\star}\,-\,R_\star^{c_n-2}, 
$$
which we claim to be larger than $R_\star.$
Indeed,
$$
\frac{1-R_\star^{c_{n+1}}}{1+R_\star}\,-\,R_\star^{c_n-2}
-R_\star=
\frac{1-R_\star^{c_{n+1}}-R_\star^{c_n-2}-R_\star^{c_n-1}-R_\star-R_\star^2}{1+R_\star}.
$$
Since~$c_i\geq4$ for all~$i$, we have
$$
\frac{1-R_\star^{c_{n+1}}-R_\star^{c_n-2}-R_\star^{c_n-1}-R_\star-R_\star^2}{1+R_\star}\geq
\frac{1-R_\star^{4}-R_\star^{3}-2R_\star^{2}-R_\star}{1+R_\star}=\frac{14-36R_\star}{1+R_\star},
$$
because~$R_\star^2=3R_\star-1$,
and so $R_\star^3=8R_\star-3$, and $R_\star^4=21R_\star-8$.
One can check that~$14-36R_\star>0$,
which implies the lemma.
\end{proof}

From this lemma by induction we have that for all $n$
$$
R_\star^{c_n-1}\left|\frac{\Sc_{n-1}(q)}{\Sc_n(q)}\right|_{C_\star}<{}R_\star^{c_n-2}.
$$
We claim that 
$$
R_\star^{c_n-2}<\frac{1-R_\star^{c_{n+1}}}{{1+R_\star}}.
$$
Indeed, under our assumptions
$$1-R_\star^{c_{n+1}}-R_\star^{c_n-1}-R_\star^{c_n-2}>
1-R_\star^{4}-R_\star^{3}-R_\star^{2},$$ which is positive. This means that
$$
\left|{\left[c_{n+1}\right]_q}\right|_{C_\star}
\;\geq\;\frac{1-R_\star^{c_{n+1}}}{{1+R_\star}}\;>\;R_\star^{c_n-2}>R_\star^{\,c_n-1}\left|\frac{\Sc_{n-1}(q)}{\Sc_n(q)}\right|_{C_\star},
$$
and we can apply Rouch\'e theorem to complete the proof.
\qed

\begin{rem}
Note that using the adopted approach
we cannot improve the assumption~\eqref{InEq}.
Indeed, assuming~$c_i\geq3$ in Lemma~\ref{LastLem2},
leads in the proof to the quantity $3-8R_\star$, which is negative.
\end{rem}

\section{Miscellaneous experiments}\label{BrolS}

Multiple computer experiments show that the situation in the $q$-deformed case
is very different from classical Markov theory \cite{Aigner}.

\begin{ex}
For~$\sqrt{3}$, one has: $\sqrt{3}=[1,\overline{1,2}]=\llbracket{}2,\overline{4}\rrbracket{}$.
The generating function of the series~$\left[\sqrt{3}\right]_q$ is
$$
\mathrm{GF}_{\left[\sqrt{3}\right]_q}
=
\frac{q^3+q^2-q-1+\sqrt{q^6 + 2q^5 + 3q^4 + 3q^2 + 2q + 1}}{2q^2},
$$
see~\cite{SVRe}.
The absolute value of the minimal root of the polynomial under the radical is $$R_{\sqrt{3}}:=R(\sqrt{3})\approx0.527756\ldots$$
which is between~$R_\star$ and~$R_{\sqrt{2}}$:
$$
R_\star<R_{\sqrt{3}}<R_{\sqrt{2}}.
$$
This example demonstrates, that, unlike classical Markov theory \cite{Aigner}, the series corresponding to~$\sqrt{2}$
``converges better'' than that of~$\sqrt{3}$.
\end{ex}

Let us also give another interesting example, which is the third ``badly approximated'' number
in Markov theory, after~$\varphi$ and  the "silver ratio" $\sqrt{2}$.

\begin{ex}
The number
$\alpha=\frac{9+\sqrt{221}}{10}=[\overline{2,2,1,1}]$, sometimes called  ``bronze ratio'', is the third most irrational number \cite{Aigner}.
In this case, the radius of convergence can be calculated explicitly in radicals.
More precisely, the radius of convergence of~$\left[\frac{9+\sqrt{221}}{10}\right]_q$ is
$$
R_\mathrm{bronze}:=R(\alpha)=
\frac{1 + \sqrt{13} - \sqrt{2\left(\sqrt{13} - 1\right)}}{4}
\quad\approx0.58069\ldots
$$
Indeed, a direct computation gives
$$
\textstyle
\left[\frac{9+\sqrt{221}}{10}\right]_q=
\frac{q^6+2q^5+3q^4+3q^3+q^2-1+\sqrt{(q^4+3q^3+5q^2+3q+1)(q^6+2q^5+3q^4+5q^3+3q^2+2q+1)(q^2-q+1)}}{2q(q^3+2q^2+q+1)}
$$
Note that~$221=13\cdot17$, the factors under the radical are $q$-versions of these numbers.
Quite remarkably, the polynomial under the radical is a palindrome polynomial
(for a general result; see~\cite{LMG}).
The radius of convergence is equal to the absolute value of the minimal root of the polynomial under the radical,
which can be found explicitly.
\end{ex}

We wonder if for other quadratic irrationals the radius of convergence 
is an algebraic number of degree~$2^n$ (like in the theory of ruler-and-compass construction)
 as it was for~$\varphi,\sqrt{2}$, and for the above example,
but we have no explicit formulas in general.

\bigbreak \noindent
{\bf Acknowledgements}.
We are very grateful to the Mathematisches Forschungsinstitut Oberwolfach
 for the hospitality during our RiP stay in summer 2020, when this project was started.
We would like to thank also Jenya Ferapontov and Sergei Tabachnikov for fruitful discussions,
and the anonymous referee for pointing out several inaccuracies in the first version of the paper.
The work of  VO  was partially supported by the ANR project ANR-19-CE40-0021.


\begin{thebibliography}{99}

\bibitem{Aigner}
  M. Aigner,
{\it Markov's Theorem and 100 Years of the Uniqueness Conjecture: A Mathematical Journey from Irrational Numbers to Perfect Matchings}. Springer, 2013.

\bibitem{Asi}
A. Bapat, L. Becker, A. Licata,
{\it $q$-deformed rational numbers and the 2-Calabi--Yau category of type $A_2$},
arXiv:2202.07613.

\bibitem{BK}
J. Bernstein, T. Khovanova,
{\it On the quantum group $\SL_q(2)$},
Comm. Math. Phys. 177 (1996), no. 3, 691--708. 

\bibitem{Conway}
J.B. Conway, Functions of One Complex Variable I. Springer-Verlag, 1978. 

\bibitem{Var}
G. Cotti, A. Varchenko,
{\it The $*$-Markov equation for Laurent polynomials},
Mosc. Math. J. {\bf 22} (2022), no. 1, 1--68..

\bibitem{Hir}
F. Hirzebruch, 
{\it Hilbert modular surfaces}, Enseign. Math. (2) 19  (1973), 183--281.

\bibitem{Hurwitz}
A. Hurwitz  {\it \"Uber die angen\"aherte Darstellung der Irrationalzahlen durch rationale Br\"uche.} Math. Annalen,  {\bf 39} (1891), 279--284.

\bibitem{Kir}
A.A. Kirillov, 
 Elements of the theory of representations. 
 Springer-Verlag, Berlin-New York, 1976.

\bibitem{LMG}
L. Leclere, S. Morier-Genoud, 
{\it The $q$-deformations in the modular group and of the real quadratic irrational numbers,}
Adv. in Appl. Math. {\bf130} (2021), Paper No. 102223, 28 pp.

\bibitem{Markov}
 A.A. Markov {\it Sur les formes quadratiques binaires ind\'efinies}.
	Math. Annalen, {\bf 15} (1879), 381-406; {\bf 17} (1880), 379-399.


\bibitem{CSS}
T. McConville, B.E. Sagan, C. Smyth,
{\it On a rank-unimodality conjecture of Morier-Genoud and Ovsienko},
Discrete Math. {\bf 344} (2021), no. 8, Paper No. 112483, 13 pp.

\bibitem{MGO} 
S.~Morier-Genoud, V.~Ovsienko, 
{\it Farey boat: continued fractions and triangulations, modular group and polygon dissections,\/}
Jahresber. Dtsch. Math.-Ver. {\bf 121} (2019), no. 2, 91--136.

\bibitem{SVRat}
S. Morier-Genoud, V. Ovsienko, 
{\it $q$-deformed rationals and $q$-continued fractions.} 
Forum Math. Sigma 8 (2020), e13, 55~pp.

\bibitem{SVRe} 
S.Morier-Genoud, V. Ovsienko,
\newblock {\it On $q$-defor\-med real numbers,}
\newblock { Exp. Math. 31 (2022), 652--660.}

\bibitem{SVJ} 
S.~Morier-Genoud, V.~Ovsienko, 
{\it Quantum numbers and $q$-deformed Conway-Coxeter friezes},
Math. Intelligencer {\bf 43} (2021), no. 2, 61--70.

\bibitem{OEIS} 
OEIS Foundation Inc., The On-Line Encyclopedia of Integer Sequences, http://oeis.org.

\bibitem{Ren} 
X. Ren,
{\it On radiuses of convergence of q-metallic numbers and related q-rational numbers},
Res. Number Theory {\bf 8} (2022), no. 3, Paper No. 37, 14 pp.

\bibitem{SV}
K. Spalding, A.P. Veselov, 
{\it Lyapunov spectrum of Markov and Euclid trees.} Nonlinearity {\bf 30} (2017), 4428--53.

\bibitem{Tit} 
E. C. Titchmarsh,  The theory of functions. 
Oxford University Press, Oxford, 1958. x+454 pp.

\end{thebibliography}
\end{document}